\documentclass[reqno,10pt]{amsart}
\usepackage{amsmath,amssymb,amsthm}
\usepackage{graphicx} 
\usepackage{xcolor}
\usepackage{tikz}

\allowdisplaybreaks

\theoremstyle{plain}
\newtheorem{theorem}{Theorem}[section]
\newtheorem{lemma}[theorem]{Lemma}
\newtheorem{proposition}[theorem]{Proposition}
\newtheorem{corollary}[theorem]{Corollary}
\newtheorem{definition}[theorem]{Definition}

\numberwithin{equation}{section}
\numberwithin{figure}{section}

\begin{document}

\title[Symmetric Jacobi Polynomials on a Triangle: Spectral Algebra]{Symmetric Jacobi Polynomials on a Triangle and Their Spectral Algebra}

\author[M. E. Marriaga and M. A.  Pi\~nar]
{Misael E. Marriaga$^*$ and Miguel A. Pi\~nar}

\address[M. E. Marriaga]{Departamento de Matem\'atica Aplicada, Ciencia e Ingenier\'ia de Materiales y
	Tecnolog\'ia Electr\'onica, Universidad Rey Juan Carlos (Spain)}
\email{misael.marriaga@urjc.es}
\address[M. A.  Pi\~nar]{
	Departamento de Matem\'{a}tica Aplicada, Facultad de Ciencias. Universidad de Granada (Spain)}
\email{mpinar@ugr.es}

\thanks{The first author (MEM) has been supported by the research project PID2021-122154NB-I00, \emph{Ortogonalidad y Aproximación con Aplicaciones en Machine Learning y Teoría de la Probabilidad}, and by the research project PID2024-155133NB-I00, \emph{Ortogonalidad, Aproximación e Integrabilidad: Aplicaciones en Procesos Estocásticos Clásicos y Cuánticos}, both funded by MICIU/AEI/10.13039/501100011033 and by ``ERDF A Way of making Europe''. The second author (MAP) has been supported by the research project PID2023-149117-NB-I00 granted by MICIU/AEI/10.13039/501100011033 and ``ERDF A Way of making Europe''.\\
$^*$\textit{Corresponding author:} Misael E. Marriaga (misael.marriaga@urjc.es)}

\date{\today}

\begin{abstract}
We study a family of symmetric orthogonal polynomials on the unit triangle
associated with the weight
\[
w_{\alpha,\gamma,\kappa}(x,y)
=
(xy)^\alpha(1-x-y)^\gamma |x-y|^{2\kappa+1},
\quad \alpha,\gamma,\kappa>-1,\quad
\alpha+\kappa>-\frac32.
\]
We construct the corresponding monic symmetric orthogonal basis on the simplex chamber
and prove that its elements are eigenfunctions of a formally self-adjoint second-order
differential operator
\(\mathcal D_1^{\alpha,\gamma,\kappa}\). A pair of adjoint ladder operators
yields a second operator of order two
\(\mathcal D_2^{\alpha,\kappa}\) with the same eigenfunctions. We also obtain
an explicit representation of the basis in terms of one-variable Jacobi
polynomials and compute its squared norms. After passing to the elementary
symmetric variables, we determine the full algebra of linear partial
differential operators with real polynomial coefficients having the transformed
polynomials as eigenfunctions. Every such operator can be written uniquely
as a polynomial in
\(\mathcal D_1^{\alpha,\gamma,\kappa}\) and
\(\mathcal D_2^{\alpha,\kappa}\); consequently, this algebra is isomorphic to
the real polynomial ring in two variables.
\end{abstract}

\subjclass[2020]{Primary 33C45; Secondary 42C05, 42C15.}

\keywords{Symmetric orthogonal polynomials; Jacobi polynomials on the simplex; ladder operators; differential operator algebras; spectral analysis.}

\maketitle

\section{Introduction}

Orthogonal polynomials in several variables form a natural meeting point of
approximation theory, special functions, and the theory of differential
operators. In contrast with the one-variable theory, where the classical
families are essentially governed by second-order Sturm--Liouville operators,
the multivariate setting is shaped by the geometry of the orthogonality domain
and by the symmetries of the underlying weight. Standard domains such as the
ball, the simplex, and product regions provide the basic models for the
theory; see, for instance, \cite{DX14}. Among these, the simplex is a
fundamental example: Jacobi-type weights lead to explicit bases in terms of
one-variable Jacobi polynomials and to natural second-order differential
operators.

Related explicit constructions have also been developed for polynomial
systems on regions bounded by lines and a parabola
\cite{Koornwinder1975,SprinkhuizenKuyper76}, and for families of weights in
two variables \cite{Xu11}. The present paper is motivated by the analogous
structural problem for a chamber of the simplex endowed with a Jacobi-type
weight and an interaction factor along the diagonal.

A second guiding theme is the study of polynomial systems that form common
eigenbases for commutative algebras of differential operators. The classification problem for bivariate orthogonal polynomials
satisfying second-order differential equations goes back to the classical
work \cite{KrallSheffer67}, while the construction of families determined by
algebraically independent partial differential operators was developed in
\cite{Ko74_I,Ko74_II}. This point of view has remained influential because it
connects the analytic structure of orthogonality with the algebraic structure
of differential operators.

Reflection symmetries provide another important source of multivariate
orthogonal polynomials. Differential-difference operators associated with
finite reflection groups \cite{Dunkl89} lead to broad classes of
multivariate special functions and orthogonal polynomial systems. Such structures are also related to
Calogero--Sutherland type models, where singular interaction terms are
naturally associated with root systems and reflection symmetries
\cite{BF97}. In this direction, generalized classical orthogonal polynomials
on the ball and on the simplex for reflection-invariant weights were studied
in \cite{Xu01}, where they are related to Dunkl-type operators and to
Calogero--Sutherland systems. Earlier work on orthogonal polynomials with
finite symmetry, including symmetry of order three and octahedral symmetry,
also illustrates how imposing symmetry on a multivariate orthogonality
problem can reveal additional differential operators and bases adapted to the
group action \cite{Dunkl84a,Dunkl84b}. For symmetric polynomial systems,
elementary symmetric functions provide a natural set of variables, as in the
construction of generalized classical families in \cite{BP23}.

The purpose of this paper is to study the Jacobi-type structure on the unit
triangle for a symmetric weight with an interaction factor along the diagonal.
More precisely, on
\[
\mathbf T^2=\{(x,y)\in\mathbb R^2:\ x\ge0,\ y\ge0,\ x+y\le1\},
\]
we consider
\[
w_{\alpha,\gamma,\kappa}(x,y)
=
(xy)^\alpha(1-x-y)^\gamma |x-y|^{2\kappa+1},
\quad
\alpha,\gamma,\kappa>-1, \quad \alpha+\kappa>-\frac32.
\]
The factor \((xy)^\alpha(1-x-y)^\gamma\) is of Jacobi type, while
\(|x-y|^{2\kappa+1}\) introduces an interaction factor along the
reflection wall \(x=y\). Since both the weight and the polynomials under
consideration are symmetric under the interchange of \(x\) and \(y\), the
orthogonality can be realized on the chamber
\[
\triangle=\{(x,y)\in\mathbf T^2:\ y<x\}.
\]
The problem is therefore to construct the underlying orthogonal basis
explicitly, determine its norms, and identify the differential operators for
which it is a system of eigenfunctions.

Our first contribution is the construction of a monic symmetric orthogonal
basis adapted to the elementary symmetric functions. We then prove that these
polynomials are eigenfunctions of a formally self-adjoint second-order
differential operator \(\mathcal D_1^{\alpha,\gamma,\kappa}\). Its formal
self-adjointness is established directly on the chamber \(\triangle\), taking
into account the boundary behavior of the weight on the three relevant
boundary components.

We also introduce a pair of adjoint operators \(D_-\) and
\(D_+^{\alpha,\kappa}\). These operators relate the orthogonal systems
associated with the weights \(w_{\alpha,\gamma,\kappa}\) and
\(w_{\alpha+1,\gamma,\kappa+1}\), and map neighboring elements between the
corresponding ordered families. In this sense they provide ladder relations
for the orthogonal basis. Their composition
\[
\mathcal D_2^{\alpha,\kappa}=D_+^{\alpha,\kappa}D_-
\]
gives a second differential operator of order two with the same
eigenfunctions. This follows the operator-theoretic philosophy of recent work
on ladder operators for bivariate generalized classical symmetric orthogonal
polynomials \cite{AMP25} and is also in line with the ladder-operator approach
developed for generalized Zernike or disk polynomials
\cite{Marriaga2025}.

A further result is an explicit Jacobi representation of the basis. We show
that the polynomials can be written as products of one-variable Jacobi
polynomials and prove directly that the resulting family is mutually
orthogonal, with closed formulas for its squared norms. These formulas
determine the normalization of the explicit basis.

Finally, after passing to the variables
\[
u=x+y,\qquad v=xy,
\]
we study the algebra of differential operators for which the transformed
polynomials form an eigenbasis. This change of variables rewrites symmetric
polynomials as ordinary bivariate polynomials on a planar domain bounded by
a parabola, a line, and an axis.
We prove that this
algebra is generated by the two second-order operators
\(\mathcal D_1^{\alpha,\gamma,\kappa}\) and
\(\mathcal D_2^{\alpha,\kappa}\).

The paper is organized as follows. Section~\ref{sec:classical-simplex}
recalls the classical Jacobi orthogonal polynomials on the simplex and fixes
the notation used throughout the paper. In Section~\ref{sec:symmetric-polys}
we introduce the symmetric weight, the chamber $\triangle$, the graded order
on symmetric monomials, and the corresponding monic orthogonal basis.
Section~\ref{sec:D1} defines the operator
$\mathcal D_{1}^{\alpha,\gamma,\kappa}$, proves its formal self-adjointness, and
shows that the symmetric orthogonal polynomials are its eigenfunctions.
Section~\ref{sec:ladder} is devoted to the operators $D_-$ and
$D_+^{\alpha,\kappa}$, their adjointness relation, their action on the
orthogonal basis, and the resulting second-order operator
$\mathcal D_{2}^{\alpha,\kappa}$. In
Section~\ref{sec:explicit-jacobi} we give an explicit representation of the
basis in terms of one-variable Jacobi polynomials and compute the squared
norms. Section~\ref{sec:uv} studies the change of variables
$u=x+y$, $v=xy$ and rewrites the main operators in these coordinates.
Finally, Section~\ref{sec:algebra} describes the algebra of differential
operators for which the transformed polynomials form an eigenbasis.

%------------------------------------------------------------------------------
\section{Classical orthogonal polynomials on the simplex}\label{sec:classical-simplex}
%------------------------------------------------------------------------------

The classical Jacobi system on the simplex will serve as the model for the
symmetric construction developed below. In particular, the explicit basis and 
the norm constants recalled in this section
are the features that we shall reproduce for the weighted chamber considered
in the rest of the paper.

Let
\[
\mathbf{T}^2=\{(x,y)\in\mathbb{R}^2:\ x\ge 0,\ y\ge 0,\ x+y\le 1\}
\]
denote the unit triangle in $\mathbb{R}^2$. For $\alpha,\beta,\gamma>-1$, consider the weight function
\[
W_{\alpha,\beta,\gamma}(x,y)=x^{\alpha}y^{\beta}(1-x-y)^{\gamma}, \qquad (x,y)\in \mathbf{T}^2,
\]
and the corresponding bilinear form
\[
\langle P,Q\rangle_{\alpha,\beta,\gamma}
=
b_{\alpha,\beta,\gamma}
\int_{\mathbf{T}^2}
P(x,y)\,Q(x,y)\,W_{\alpha,\beta,\gamma}(x,y)\,dx\,dy,
\]
where
\[
b_{\alpha,\beta,\gamma}
=
\left(
\int_{\mathbf{T}^2}
W_{\alpha,\beta,\gamma}(x,y)\,dx\,dy
\right)^{-1}
=
\frac{\Gamma(\alpha+\beta+\gamma+3)}
{\Gamma(\alpha+1)\Gamma(\beta+1)\Gamma(\gamma+1)}.
\]

A standard orthogonal basis on the triangle is obtained by coupling Jacobi polynomials in a triangular coordinate system. We therefore begin by recalling the one-variable family that underlies the construction. The Jacobi polynomial of degree $n$ is orthogonal on $[-1,1]$ with respect to the weight
\[
w_{\alpha,\beta}(t)=(1-t)^{\alpha}(1+t)^{\beta}, \qquad \alpha,\beta>-1.
\]
It is typically defined as \cite[(4.1.1)]{Szego75}
\[
P_n^{(\alpha,\beta)}(t)
=
\frac{1}{n!}
\sum_{k=0}^n
\binom{n}{k}
(k+\alpha+1)_{n-k}
(n+\alpha+\beta+1)_k
\left(\frac{t-1}{2}\right)^k,
\]
where, as usual,
\[
(a)_0=1, \qquad (a)_k=a(a+1)\cdots(a+k-1), \quad k\ge 1,
\]
denotes the Pochhammer symbol. These polynomials satisfy the normalization
\[
P_n^{(\alpha,\beta)}(1)=\frac{(\alpha+1)_n}{n!}.
\]

The following result gives one of the classical orthogonal bases on the triangle.

\begin{proposition}(\cite[p.~35]{DX14})
For $n\ge 0$, define
\[
P_{n,m}^{(\alpha,\beta,\gamma)}(x,y)
=
P_{n-m}^{(\beta_m,\alpha)}(2x-1)\,(1-x)^m\,
P_m^{(\gamma,\beta)}\!\left(\frac{2y}{1-x}-1\right),
\qquad 0\le m\le n,
\]
where
\[
\beta_m=\beta+\gamma+2m+1.
\]
Then
\[
\left\{
P_{n,m}^{(\alpha,\beta,\gamma)}(x,y):\ n\ge 0,\ 0\le m\le n
\right\}
\]
is a mutually orthogonal polynomial system with respect to $\langle\cdot,\cdot\rangle_{\alpha,\beta,\gamma}$. Moreover,
\[
\left\langle
P_{n,m}^{(\alpha,\beta,\gamma)},
P_{k,j}^{(\alpha,\beta,\gamma)}
\right\rangle_{\alpha,\beta,\gamma}
=
H_{n,m}^{(\alpha,\beta,\gamma)}\,
\delta_{n,k}\delta_{m,j},
\]
where
\begin{align*}
H_{n,m}^{(\alpha,\beta,\gamma)}
&=
\frac{
(\alpha+1)_{n-m}\,
(\beta+1)_m\,
(\gamma+1)_m\,
(\beta+\gamma+2)_{n+m}
}{
(n-m)!\,m!\,(\beta+\gamma+2)_m\,(\alpha+\beta+\gamma+3)_{n+m}
}\\
&\quad \times
\frac{
(n+m+\alpha+\beta+\gamma+2)\,(m+\beta+\gamma+1)
}{
(2n+\alpha+\beta+\gamma+2)\,(2m+\beta+\gamma+1)
}.
\end{align*}
For compactness, this formula is understood after cancelling the common
factors at exceptional parameter values.
\end{proposition}

%------------------------------------------------------------------------------
\section{Symmetric orthogonal polynomials}\label{sec:symmetric-polys}
%------------------------------------------------------------------------------

We now turn to the symmetric orthogonal families that are the main object of this paper. Here, symmetry refers specifically to invariance under the interchange of the variables $x$ and $y$. This requirement will be imposed both on the polynomials and on the underlying weight. Consequently, starting from the classical weight on the triangle, we are naturally led to the specialization in which the exponents of $x$ and $y$ coincide. For this reason, we consider the symmetric specialization $\alpha=\beta$ and define
\[
w_{\alpha,\gamma,\kappa}(x,y)
=
(xy)^{\alpha}(1-x-y)^{\gamma}|x-y|^{2\kappa+1},
\quad \alpha,\gamma,\kappa>-1,\quad \alpha+\kappa>-\frac32,
\]
on the simplex $\mathbf{T}^2$. Thus, the parameter $\kappa$ controls the interaction term along the diagonal $x=y$, while $\alpha$ and $\gamma$ govern the classical Jacobi-type behavior at the boundary of the simplex.

Since the weight satisfies
\[
w_{\alpha,\gamma,\kappa}(y,x)=w_{\alpha,\gamma,\kappa}(x,y),
\]
the orthogonality can be realized on the set
\[
\triangle=\{(x,y)\in \mathbf{T}^2:\ y<x\}.
\]

Let $\mathbb{N}_0^2$ denote the set of pairs of nonnegative integers. We write
\[
(m,l)\prec(n,k)
\quad\Longleftrightarrow\quad
\bigl(m+2l<n+2k\bigr)\ \text{or}\ 
\bigl(m+2l=n+2k\ \text{and}\ l<k\bigr),
\]
and use $\preceq$ for the associated non-strict ordering. Thus, the
ordering is graded by the weighted sum of the indices, where the second
index has weight two, and, within each fixed weighted sum, by the second
index.

By the fundamental theorem of symmetric polynomials, every symmetric
polynomial $p(x,y)$ with real coefficients can be written uniquely as a
polynomial in the elementary symmetric functions
\[
x+y,\qquad xy.
\]
Consequently, every symmetric polynomial $p(x,y)$ can be written uniquely
as a finite linear combination of the elementary-symmetric monomials
\[
\{(x+y)^n(xy)^k:\ n,k\in\mathbb{N}_0\}.
\]
If, in this expansion, $(n,k)$ is the largest pair, with respect to
$\preceq$, appearing with nonzero coefficient, then we call $(n,k)$ the
\emph{symmetric degree} of $p(x,y)$. The ordinary total degree associated
with the monomial $(x+y)^n(xy)^k$ is $n+2k$.

A nonzero symmetric polynomial $p(x,y)$ of symmetric degree $(n,k)$ is said to be \emph{orthogonal} with respect to $w_{\alpha,\gamma,\kappa}$ if
\[
\int_{\triangle}
p(x,y)\,q(x,y)\,w_{\alpha,\gamma,\kappa}(x,y)\,dx\,dy=0
\]
for every symmetric polynomial $q(x,y)$ whose symmetric degree $(m,l)$ satisfies $(m,l)\prec(n,k)$. 

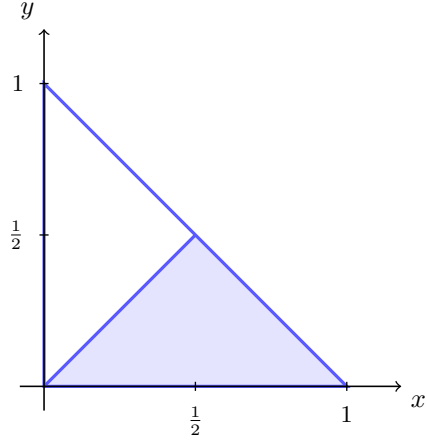
\begin{figure}[t]
	\centering
	\begin{tikzpicture}[scale=4]
		
		% --- Fill the region \triangle = {(x,y) in T^2 : y < x} ---
		\fill[blue!30, fill opacity=0.35] (0,0) -- (1,0) -- (0.5,0.5) -- cycle;
		
		% --- Boundary of the simplex T^2 ---
		\draw[line width=1.2pt,blue!65] (0,0) -- (1,0) -- (0,1) -- cycle;
		
		% --- The separating line y = x inside the simplex ---
		\draw[line width=1.2pt,blue!65] (0,0) -- (0.5,0.5);
		
		% --- Axes (solid black) ---
		\draw[->,black,line width=0.6pt] (-0.08,0) -- (1.18,0) node[below right] {$x$};
		\draw[->,black,line width=0.6pt] (0,-0.08) -- (0,1.18) node[above left] {$y$};
		
		% --- Tick marks and labels ---
		\foreach \x in {0.5,1} {
			\draw[black,line width=0.5pt] (\x,0.015) -- (\x,-0.015);
		}
		\foreach \y in {0.5,1} {
			\draw[black,line width=0.5pt] (0.015,\y) -- (-0.015,\y);
		}
		
		\node[below=4pt] at (0.5,0) {\small $\frac12$};
		\node[below=4pt] at (1,0) {\small $1$};
		\node[left=4pt] at (0,0.5) {\small $\frac12$};
		\node[left=4pt] at (0,1) {\small $1$};
		
	\end{tikzpicture}
	\caption{The region $\triangle=\{(x,y)\in \mathbf{T}^2:\ y< x\}$.}
	\label{fig:region-simplex-y<x}
\end{figure}

Arrange the elementary-symmetric monomials according to the order
$\preceq$. The first few monomials are shown below:
\[
\begin{array}{c|l}
n+2k & \text{basis elements}\\
\hline
0 & 1\\[2pt]
1 & x+y\\[2pt]
2 & (x+y)^2,\quad xy\\[2pt]
3 & (x+y)^3,\quad (x+y)xy\\[2pt]
4 & (x+y)^4,\quad (x+y)^2xy,\quad (xy)^2\\[2pt]
5 & (x+y)^5,\quad (x+y)^3xy,\quad (x+y)(xy)^2 .
\end{array}
\]
Since \(\alpha,\gamma,\kappa>-1\) and \(\alpha+\kappa>-\frac32\),
the weight $w_{\alpha,\gamma,\kappa}$ is integrable on $\triangle$ and
strictly positive in its interior. Hence the corresponding weighted
inner product is finite and positive definite on symmetric polynomials.
Applying the Gram--Schmidt procedure to the ordered
elementary-symmetric monomials with respect to the weight
\(w_{\alpha,\gamma,\kappa}(x,y)\) yields an orthogonal basis of the
space of symmetric polynomials,
\[
\bigl\{p^{\alpha,\gamma,\kappa}_{n,k}(x,y):
n,k\in\mathbb{N}_0\bigr\},
\]
where \(p^{\alpha,\gamma,\kappa}_{n,k}\) has symmetric degree \((n,k)\)
and leading term \((x+y)^n(xy)^k\).

This ordered orthogonal system provides the polynomial framework on which the
spectral analysis of the weighted chamber will be developed in the following
sections.

%------------------------------------------------------------------------------
\section{The differential operator $\mathcal{D}_{1}^{\alpha,\gamma,\kappa}$}\label{sec:D1}
%------------------------------------------------------------------------------

We now introduce the first second-order differential operator associated with
the symmetric orthogonal basis. This operator provides the main spectral
structure for the family considered in this paper.

We will show that, in general, the symmetric orthogonal polynomials
$p_{n,k}^{\alpha,\gamma,\kappa}(x,y)$ are eigenfunctions of a second-order
differential operator with rational coefficients, denoted by
$\mathcal{D}_{1}^{\alpha,\gamma,\kappa}$. To define it, we set
\[
z=1-x-y,
\qquad
\partial_z:=\partial_y-\partial_x,
\]
where $\partial_x:=\frac{\partial}{\partial x}$ and
$\partial_y:=\frac{\partial}{\partial y}$.
Then
\begin{equation*}%\label{D1_agk_1}
\mathcal{D}_{1}^{\alpha,\gamma,\kappa}
:=
\frac{1}{w_{\alpha,\gamma,\kappa}}
\Bigl[
\partial_x\bigl(xz\,w_{\alpha,\gamma,\kappa}\,\partial_x\bigr)
 +\partial_y\bigl(yz\,w_{\alpha,\gamma,\kappa}\,\partial_y\bigr)
 +\partial_z\bigl(xy\,w_{\alpha,\gamma,\kappa}\,\partial_z\bigr)
\Bigr].
\end{equation*}
 Expanding this expression, we
obtain the alternative representation
\begin{equation}\label{D1_agk_2}
\resizebox{0.9\textwidth}{!}{$
\begin{aligned}
\mathcal{D}_{1}^{\alpha,\gamma,\kappa}
&=
x(1-x)\partial_x^2
-2xy\,\partial_x\partial_y
+y(1-y)\partial_y^2
\\
&\quad
+
\left[\alpha+1-(2\alpha+\gamma+2\kappa+4)x\right]\partial_x
+
\left[\alpha+1-(2\alpha+\gamma+2\kappa+4)y\right]\partial_y
\\
&\quad
+
(2\kappa+1)\frac{x\partial_x-y\partial_y}{x-y}.
\end{aligned}
$}
\end{equation}
The last term has an apparent singularity on the diagonal $x=y$. However,
when the operator acts on a symmetric polynomial $p(x,y)$, the numerator is divisible by $x-y$, and therefore
$\mathcal{D}_{1}^{\alpha,\gamma,\kappa}p$ is again a polynomial.

\begin{proposition}\label{prop:D1_agk_selfadjoint}
Let $f(x,y)$ and $g(x,y)$ be symmetric polynomials. Then
$\mathcal{D}_{1}^{\alpha,\gamma,\kappa}$ is formally self-adjoint with respect to
$w_{\alpha,\gamma,\kappa}$ on $\triangle$, that is,
\[
\resizebox{\textwidth}{!}{$\displaystyle\int_{\triangle}
\mathcal{D}_{1}^{\alpha,\gamma,\kappa}[f(x,y)]\,g(x,y)\,
w_{\alpha,\gamma,\kappa}(x,y)\,dx\,dy
=
\int_{\triangle}
f(x,y)\,\mathcal{D}_{1}^{\alpha,\gamma,\kappa}[g(x,y)]\,
w_{\alpha,\gamma,\kappa}(x,y)\,dx\,dy .$}
\]
\end{proposition}

\begin{proof}
Write
\[
z=1-x-y,
\qquad
w=w_{\alpha,\gamma,\kappa}(x,y).
\]
Since $\partial_z=\partial_y-\partial_x$,
$\mathcal{D}_{1}^{\alpha,\gamma,\kappa}$ can be rewritten in the ordinary
divergence form
\begin{equation}\label{D1_agk_divergence_form}
\mathcal{D}_{1}^{\alpha,\gamma,\kappa} f
=
\frac{1}{w}
\left[
\partial_x(wA_f)+\partial_y(wB_f)
\right],
\end{equation}
where
\[
A_f=x(1-x)\partial_x f-xy\,\partial_y f,
\qquad
B_f=y(1-y)\partial_y f-xy\,\partial_x f.
\]

For $\varepsilon>0$, set
\[
\triangle_\varepsilon
=
\{(x,y)\in\triangle:\ y>\varepsilon,\ z>\varepsilon,\ x-y>\varepsilon\}.
\]
On $\triangle_\varepsilon$ all coefficients of
$\mathcal{D}_{1}^{\alpha,\gamma,\kappa} f$ are regular. Multiplying
\eqref{D1_agk_divergence_form} by $g w$ and using Green's theorem gives
\begin{equation}\label{D1_agk_green}
\resizebox{0.90\textwidth}{!}{$
\begin{aligned}
&\int_{\triangle_\varepsilon}
\mathcal{D}_{1}^{\alpha,\gamma,\kappa}[f]\,g\,w\,dx\,dy
=
\int_{\partial\triangle_\varepsilon}
gwA_f\,dy-gwB_f\,dx
\\
&\quad
-
\int_{\triangle_\varepsilon}
w\left[
x(1-x)\partial_x f\,\partial_x g
-xy\,\partial_y f\,\partial_x g
-xy\,\partial_x f\,\partial_y g
+y(1-y)\partial_y f\,\partial_y g
\right]dx\,dy .
\end{aligned}
$}
\end{equation}

Lemma~\ref{lem:D1-boundary-estimates} shows that the boundary term in
\eqref{D1_agk_green} tends to zero as $\varepsilon\to0^+$. The integrand
in the area integral is also integrable under the standing parameter
assumptions. Letting $\varepsilon\to0^+$, we obtain
\begin{align*}
\int_{\triangle}
\mathcal{D}_{1}^{\alpha,\gamma,\kappa}[f]\,g\,w\,dx\,dy
&=
-
\int_{\triangle}
w\Bigl[
x(1-x)\partial_x f\,\partial_x g
-xy\,\partial_y f\,\partial_x g
\\
&\qquad\qquad
-xy\,\partial_x f\,\partial_y g
+y(1-y)\partial_y f\,\partial_y g
\Bigr]dx\,dy .
\end{align*}
The expression on the right-hand side is symmetric in $f$ and $g$.
Therefore,
\[
\int_{\triangle}
\mathcal{D}_{1}^{\alpha,\gamma,\kappa}[f]\,g\,w\,dx\,dy
=
\int_{\triangle}
f\,\mathcal{D}_{1}^{\alpha,\gamma,\kappa}[g]\,w\,dx\,dy ,
\]
which proves the result.
\end{proof}

The following triangular action is obtained by a direct computation from
\eqref{D1_agk_2}.

\begin{proposition}\label{prop:D1_triangular_action}
For $n,k\in\mathbb{N}_0$, 
\begin{equation*}%\label{eq:D1_triangular_action}
\begin{aligned}
\mathcal{D}_{1}^{\alpha,\gamma,\kappa}
\left[(x+y)^n(xy)^k\right] & =
\lambda_{n,k}^{\alpha,\gamma,\kappa}(x+y)^n(xy)^k\\
&\qquad
+
n(n+4k+2\alpha+2\kappa+2)(x+y)^{n-1}(xy)^k\\
&\qquad
+
k(k+\alpha)(x+y)^{n+1}(xy)^{k-1},
\end{aligned}
\end{equation*}
where terms with negative exponents are omitted, and
\begin{equation}\label{eq:lambda_agk}
\lambda_{n,k}^{\alpha,\gamma,\kappa}
=
-(n+2k)(n+2k+2\alpha+2\kappa+\gamma+3).
\end{equation}
\end{proposition}

In particular, if a symmetric polynomial $p(x,y)$ has the form
\[
p(x,y)
=
c_{n,k}(x+y)^n(xy)^k
+
\text{a symmetric polynomial of degree lower than }(n,k),
\]
then Proposition~\ref{prop:D1_triangular_action} implies that
\begin{align*}
\mathcal{D}_{1}^{\alpha,\gamma,\kappa}[p(x,y)]
&=
\lambda_{n,k}^{\alpha,\gamma,\kappa}
c_{n,k}(x+y)^n(xy)^k
\\
&\quad
+
\text{a symmetric polynomial of degree lower than }(n,k).
\end{align*}
Therefore, $\mathcal{D}_{1}^{\alpha,\gamma,\kappa}$ does not increase the
symmetric degree.

\begin{theorem}\label{th:D1_eigen}
Let $\{p_{n,k}^{\alpha,\gamma,\kappa}(x,y):\ n,k\in\mathbb{N}_0\}$ be the orthogonal basis
constructed above. Then
\[
\mathcal{D}_{1}^{\alpha,\gamma,\kappa}
\big[p_{n,k}^{\alpha,\gamma,\kappa}(x,y)\big]
=
\lambda_{n,k}^{\alpha,\gamma,\kappa}\,
p_{n,k}^{\alpha,\gamma,\kappa}(x,y),
\]
where $\lambda_{n,k}^{\alpha,\gamma,\kappa}$ is given in
\eqref{eq:lambda_agk}.
\end{theorem}

\begin{proof}
Since $p_{n,k}^{\alpha,\gamma,\kappa}$ has symmetric degree $(n,k)$, the
triangular action in Proposition~\ref{prop:D1_triangular_action} gives
\[
\mathcal{D}_{1}^{\alpha,\gamma,\kappa}
\big[p_{n,k}^{\alpha,\gamma,\kappa}\big]
=
\lambda_{n,k}^{\alpha,\gamma,\kappa}
p_{n,k}^{\alpha,\gamma,\kappa}
+
\sum_{(m,l)\prec(n,k)}
c_{m,l}\,
p_{m,l}^{\alpha,\gamma,\kappa}
\]
for some coefficients $c_{m,l}$. For each $(m,l)\prec(n,k)$, set
\[
h_{m,l}^{\alpha,\gamma,\kappa}
=
\int_{\triangle}
\bigl(p_{m,l}^{\alpha,\gamma,\kappa}(x,y)\bigr)^2
w_{\alpha,\gamma,\kappa}(x,y)\,dx\,dy .
\]
Since $\alpha,\gamma,\kappa>-1$ and \(\alpha+\kappa>-\frac32\), the weight
$w_{\alpha,\gamma,\kappa}$ is integrable on $\triangle$ and is strictly
positive in its interior. Moreover,
$p_{m,l}^{\alpha,\gamma,\kappa}$ is a nonzero polynomial, and hence it
cannot vanish on any nonempty open subset of $\triangle$. It follows that $0<h_{m,l}^{\alpha,\gamma,\kappa}<\infty$. Therefore,
\begin{align*}
c_{m,l}\,h_{m,l}^{\alpha,\gamma,\kappa}
&=
\int_{\triangle}
\mathcal{D}_{1}^{\alpha,\gamma,\kappa}
\big[p_{n,k}^{\alpha,\gamma,\kappa}(x,y)\big]\,
p_{m,l}^{\alpha,\gamma,\kappa}(x,y)\,
w_{\alpha,\gamma,\kappa}(x,y)\,dx\,dy .
\end{align*}
By Proposition~\ref{prop:D1_agk_selfadjoint},
\begin{align*}
c_{m,l}\,h_{m,l}^{\alpha,\gamma,\kappa}
&=
\int_{\triangle}
p_{n,k}^{\alpha,\gamma,\kappa}(x,y)\,
\mathcal{D}_{1}^{\alpha,\gamma,\kappa}
\big[p_{m,l}^{\alpha,\gamma,\kappa}(x,y)\big]\,
w_{\alpha,\gamma,\kappa}(x,y)\,dx\,dy .
\end{align*}
Since $(m,l)\prec(n,k)$,
Proposition~\ref{prop:D1_triangular_action} implies that $\mathcal{D}_{1}^{\alpha,\gamma,\kappa}
\big[p_{m,l}^{\alpha,\gamma,\kappa}\big]$ is a linear combination of basis elements
$p_{r,s}^{\alpha,\gamma,\kappa}$ with
$(r,s)\preceq(m,l)\prec(n,k)$. Hence, by orthogonality, the last integral
vanishes. Therefore
\[
c_{m,l}\,h_{m,l}^{\alpha,\gamma,\kappa}=0.
\]
Since $h_{m,l}^{\alpha,\gamma,\kappa}>0$, we conclude that
$c_{m,l}=0$ for all $(m,l)\prec(n,k)$. This proves the result.
\end{proof}

%------------------------------------------------------------------------------
\section{Lowering and raising operators}\label{sec:ladder}
%------------------------------------------------------------------------------

The second-order operator obtained in the previous section gives one part of
the spectral structure of the family. We now introduce a pair of adjoint
operators that relates neighboring orthogonal systems and leads to a second
operator acting diagonally on the same basis.

We introduce the operators
\begin{equation}\label{Dminus_agk}
D_{-}
:=
\frac{\partial_y-\partial_x}{x-y},
\end{equation}
and
\begin{equation}\label{Dplus_agk}
D_{+}^{\alpha,\kappa}
:=
\frac{-1}{(xy)^\alpha (x-y)^{2\kappa}}
D_{-}\left[
(xy)^{\alpha+1}(x-y)^{2\kappa+2}\,\cdot
\right].
\end{equation}
Note that $D_{-}$ does not depend on the parameters
$\alpha$ and $\kappa$. Although $D_{-}$ has an apparent
singularity on the diagonal $x=y$, it is well defined on symmetric
polynomials, since $(\partial_y-\partial_x)f$ is divisible by $x-y$
whenever $f$ is symmetric.

The operator $D_{+}^{\alpha,\kappa}$ admits the representation
\begin{equation}\label{Dplus_agk_expanded}
D_{+}^{\alpha,\kappa}
=
xy(x-y)(\partial_x-\partial_y)
+
4(\kappa+1)xy
-
(\alpha+1)(x-y)^2 .
\end{equation}

We now show the adjoint relation satisfied by the operators $D_{-}$ and
$D_{+}^{\alpha,\kappa}$.

\begin{proposition}\label{prop:D_adjoints_agk}
Let $f(x,y)$ and $g(x,y)$ be symmetric polynomials. Then
\begin{align*}
\int_{\triangle}
D_{-}[f(x,y)]\,g(x,y)&\,
w_{\alpha+1,\gamma,\kappa+1}(x,y)\,dx\,dy
\notag\\
& =
\int_{\triangle}
f(x,y)\,D_{+}^{\alpha,\kappa}[g(x,y)]\,
w_{\alpha,\gamma,\kappa}(x,y)\,dx\,dy .
%\label{eq:D_adjoints_agk}
\end{align*}
\end{proposition}

\begin{proof}
Write $z=1-x-y$. Since $(\partial_x-\partial_y)z=0$, it follows from
\eqref{Dplus_agk} that
\begin{equation}\label{eq:Dplus_weight_identity}
D_{+}^{\alpha,\kappa}[g]\,w_{\alpha,\gamma,\kappa}(x,y)
=
(\partial_x-\partial_y)
\left[
g\,(xy)^{\alpha+1}z^\gamma(x-y)^{2\kappa+2}
\right].
\end{equation}

For $\varepsilon>0$, set
\[
\triangle_\varepsilon
=
\{(x,y)\in\triangle:\ y>\varepsilon,\ z>\varepsilon,\ x-y>\varepsilon\}.
\]
On $\triangle_\varepsilon$ all coefficients are regular. Since $x-y>0$
on $\triangle$, we have
\[
w_{\alpha+1,\gamma,\kappa+1}(x,y)
=
(xy)^{\alpha+1}z^\gamma(x-y)^{2\kappa+3}.
\]
Using \eqref{Dminus_agk}, we obtain
\begin{align*}
\int_{\triangle_\varepsilon}
D_{-}[f]\,g\,w_{\alpha+1,\gamma,\kappa+1}\,dx\,dy
&=
\int_{\triangle_\varepsilon}
(\partial_y-\partial_x)f\,
g\,(xy)^{\alpha+1}z^\gamma(x-y)^{2\kappa+2}\,dx\,dy .
\end{align*}
Applying Green's theorem on $\triangle_\varepsilon$ gives
\begin{equation}
\begin{aligned}
\int_{\triangle_\varepsilon}
D_{-}[f]\,g&\,w_{\alpha+1,\gamma,\kappa+1}\,dx\,dy \\
&=
\int_{\triangle_\varepsilon}
f\,(\partial_x-\partial_y)
\left[
g\,(xy)^{\alpha+1}z^\gamma(x-y)^{2\kappa+2}
\right]dx\,dy
+
R_\varepsilon ,
\end{aligned}
\label{eq:Dadj_green}
\end{equation}
where $R_\varepsilon$ is the boundary contribution. By
Lemma~\ref{lem:ladder-boundary-estimates},
\[
R_\varepsilon\longrightarrow0
\qquad
\text{as }\varepsilon\to0^+.
\]

The weight $w_{\alpha,\gamma,\kappa}$ is integrable since
$\alpha,\gamma,\kappa>-1$ and $\alpha+\kappa>-\frac32$. The shifted
weight $w_{\alpha+1,\gamma,\kappa+1}$ is also integrable, since
\[
(\alpha+1)+(\kappa+1)>-\frac32
\]
follows automatically from \(\alpha,\kappa>-1\). Thus the two
integrals in \eqref{eq:Dadj_green} are finite. Letting
$\varepsilon\to0^+$ in \eqref{eq:Dadj_green} and using
\eqref{eq:Dplus_weight_identity}, we obtain
\[
\int_{\triangle}
D_{-}[f]\,g\,w_{\alpha+1,\gamma,\kappa+1}\,dx\,dy
=
\int_{\triangle}
f\,D_{+}^{\alpha,\kappa}[g]\,
w_{\alpha,\gamma,\kappa}\,dx\,dy,
\]
as claimed.
\end{proof}

The following identities follow by a direct computation from
\eqref{Dminus_agk} and \eqref{Dplus_agk_expanded}.

\begin{proposition}\label{prop:D_ladder_monomial_action}
For $n,k\in\mathbb{N}_0$,
\begin{equation*}%\label{eq:Dminus_monomial_action}
D_{-}\left[(x+y)^n(xy)^k\right]
=
k(x+y)^n(xy)^{k-1},
\end{equation*}
where the right-hand side is understood to be zero when $k=0$. Moreover,
\begin{equation}\label{eq:Dplus_monomial_action}
\begin{aligned}
D_{+}^{\alpha,\kappa}
\left[(x+y)^n(xy)^k\right]
&=
-(k+\alpha+1)(x+y)^{n+2}(xy)^k
\\
&\quad
+
4(k+\alpha+\kappa+2)(x+y)^n(xy)^{k+1}.
\end{aligned}
\end{equation}
\end{proposition}

Consequently, $D_{-}$ lowers the leading symmetric degree according to $(n,k)\mapsto(n,k-1)$, whereas, with respect to the order $\preceq$, $D_{+}^{\alpha,\kappa}$
raises the leading symmetric degree according to $(n,k)\mapsto(n,k+1)$. Indeed, the two terms in \eqref{eq:Dplus_monomial_action} have the same
weighted sum of indices, but the second has larger second index.

\begin{theorem}\label{th:ladderD_agk}
Let $\{p_{n,k}^{\alpha,\gamma,\kappa}(x,y):\ n,k\in\mathbb{N}_0\}$ be the orthogonal basis constructed above. Then
\[
D_{-}\big[p_{n,k}^{\alpha,\gamma,\kappa}(x,y)\big]
=
k\,p_{n,k-1}^{\alpha+1,\gamma,\kappa+1}(x,y),
\]
where the right-hand side is understood to be zero when $k=0$, and
\[
D_{+}^{\alpha,\kappa}
\big[p_{n,k}^{\alpha+1,\gamma,\kappa+1}(x,y)\big]
=
4(k+\alpha+\kappa+2)\,
p_{n,k+1}^{\alpha,\gamma,\kappa}(x,y).
\]
\end{theorem}

\begin{proof}
We first consider the case \(k\geq1\). By
Proposition~\ref{prop:D_ladder_monomial_action},
\[
D_{-}\big[p_{n,k}^{\alpha,\gamma,\kappa}\big]
=
k\,p_{n,k-1}^{\alpha+1,\gamma,\kappa+1}
+
\sum_{(m,l)\prec(n,k-1)}
c_{m,l}\,p_{m,l}^{\alpha+1,\gamma,\kappa+1}.
\]
For \(m,l\in\mathbb N_0\), set
\[
h_{m,l}^{\alpha+1,\gamma,\kappa+1}
=
\int_{\triangle}
\bigl(p_{m,l}^{\alpha+1,\gamma,\kappa+1}(x,y)\bigr)^2
w_{\alpha+1,\gamma,\kappa+1}(x,y)\,dx\,dy .
\]
Since
\[
\alpha+1,\gamma,\kappa+1>-1,
\qquad
(\alpha+1)+(\kappa+1)>-\frac32,
\]
the shifted weight is integrable and strictly positive in the interior
of $\triangle$. Hence
\[
0<h_{m,l}^{\alpha+1,\gamma,\kappa+1}<\infty.
\]
Therefore,
\[
\begin{aligned}
c_{m,l}h_{m,l}^{\alpha+1,\gamma,\kappa+1}
=
\int_{\triangle}
D_{-}\big[p_{n,k}^{\alpha,\gamma,\kappa}(x,y)\big]\,
p_{m,l}^{\alpha+1,\gamma,\kappa+1}(x,y) w_{\alpha+1,\gamma,\kappa+1}(x,y)\,dx\,dy .
\end{aligned}
\]
Using Proposition~\ref{prop:D_adjoints_agk}, we get
\[
\begin{aligned}
c_{m,l}h_{m,l}^{\alpha+1,\gamma,\kappa+1}
=
\int_{\triangle}
p_{n,k}^{\alpha,\gamma,\kappa}(x,y)\,
D_{+}^{\alpha,\kappa}
\big[p_{m,l}^{\alpha+1,\gamma,\kappa+1}(x,y)\big]
w_{\alpha,\gamma,\kappa}(x,y)\,dx\,dy .
\end{aligned}
\]
Since $(m,l)\prec(n,k-1)$, Proposition~\ref{prop:D_ladder_monomial_action}
implies that $D_{+}^{\alpha,\kappa}
\big[p_{m,l}^{\alpha+1,\gamma,\kappa+1}\big]$ is a symmetric polynomial of degree lower than $(n,k)$. Hence the last
integral vanishes by orthogonality. Therefore $c_{m,l}=0$.

Now let \(k=0\). Since
\(p_{n,0}^{\alpha,\gamma,\kappa}\) has leading term \((x+y)^n\), it can
be written as
\[
p_{n,0}^{\alpha,\gamma,\kappa}
=
(x+y)^n
+
\sum_{(a,b)\prec(n,0)}
d_{a,b}(x+y)^a(xy)^b .
\]
By Proposition~\ref{prop:D_ladder_monomial_action}, the leading term is
annihilated by \(D_-\), as are all the terms with \(b=0\). If \(b\geq1\),
then
\[
D_-\left[(x+y)^a(xy)^b\right]
=
b(x+y)^a(xy)^{b-1}.
\]
Since \((a,b)\prec(n,0)\) implies \(a+2b<n\), every nonzero term in
\(D_-[p_{n,0}^{\alpha,\gamma,\kappa}]\) has weighted degree
\(
a+2(b-1)\leq n-3
\). We may therefore write
\[
D_-\big[p_{n,0}^{\alpha,\gamma,\kappa}\big]
=
\sum_{m+2l\leq n-3}
c_{m,l}\,
p_{m,l}^{\alpha+1,\gamma,\kappa+1}.
\]
For each term in this expansion, Proposition~\ref{prop:D_adjoints_agk}
gives
\[
\begin{aligned}
c_{m,l}h_{m,l}^{\alpha+1,\gamma,\kappa+1}
=
\int_{\triangle}
p_{n,0}^{\alpha,\gamma,\kappa}(x,y)\,
D_+^{\alpha,\kappa}
\big[p_{m,l}^{\alpha+1,\gamma,\kappa+1}(x,y)\big]
w_{\alpha,\gamma,\kappa}(x,y)\,dx\,dy .
\end{aligned}
\]
By Proposition~\ref{prop:D_ladder_monomial_action}, the polynomial
\(D_+^{\alpha,\kappa}
[p_{m,l}^{\alpha+1,\gamma,\kappa+1}]\) has degree at most
\((m,l+1)\), whose weighted degree satisfies
\[
m+2(l+1)\leq n-1<n.
\]
The integral therefore vanishes by orthogonality. Hence
\(c_{m,l}=0\) for every \(m,l\), and
\[
D_-\big[p_{n,0}^{\alpha,\gamma,\kappa}\big]=0.
\]

This proves the first identity. The second follows analogously, using
\eqref{eq:Dplus_monomial_action} and Proposition~\ref{prop:D_adjoints_agk}.
\end{proof}

We define the second-order operator
\begin{equation*}%\label{D2_agk}
\mathcal{D}_{2}^{\alpha,\kappa}
:=
D_{+}^{\alpha,\kappa}D_{-}.
\end{equation*}

\begin{corollary}\label{coro:D2_agk}
Let $\{p_{n,k}^{\alpha,\gamma,\kappa}(x,y):\ n,k\in\mathbb{N}_0\}$ be the orthogonal basis constructed above. Then
\[
\mathcal{D}_{2}^{\alpha,\kappa}
\big[p_{n,k}^{\alpha,\gamma,\kappa}(x,y)\big]
=
4k(k+\alpha+\kappa+1)\,
p_{n,k}^{\alpha,\gamma,\kappa}(x,y).
\]
\end{corollary}

Thus the ladder construction produces a second spectral operator associated
with the same orthogonal basis. This operator will be used later in the
description of the differential-operator algebra.

%------------------------------------------------------------------------------
\section{An explicit Jacobi basis}\label{sec:explicit-jacobi}
%------------------------------------------------------------------------------

The preceding sections characterize the orthogonal basis through its
differential operators. We now give an explicit Jacobi representation of the
same basis.

\begin{definition}%\label{def:S_nk_agk}
For $n,k\in\mathbb{N}_0$, define
\begin{equation}\label{S_nk_agk}
S_{n,k}^{\alpha,\gamma,\kappa}(x,y)
=
(x+y)^{2k}
P_n^{(2\alpha+2\kappa+4k+2,\gamma)}(1-2x-2y)
P_k^{(\kappa,\alpha)}
\left(1-2\frac{(x-y)^2}{(x+y)^2}\right).
\end{equation}
\end{definition}

Although the last factor involves the quotient $(x-y)^2/(x+y)^2$, the factor $(x+y)^{2k}$ clears the denominator. Hence
$S_{n,k}^{\alpha,\gamma,\kappa}(x,y)$ is a symmetric polynomial in $x$ and
$y$.

\begin{theorem}\label{th:S_nk_orthogonal}
For $n,k\in\mathbb{N}_0$, let
$S_{n,k}^{\alpha,\gamma,\kappa}(x,y)$ be the polynomials defined in
\eqref{S_nk_agk}. Then
\[
\left\{
S_{n,k}^{\alpha,\gamma,\kappa}(x,y):\ n,k\in\mathbb{N}_0
\right\}
\]
is a mutually orthogonal polynomial system with respect to
$w_{\alpha,\gamma,\kappa}$ on $\triangle$. Moreover,
\[
b_{\alpha,\gamma,\kappa}
\int_{\triangle}
S_{n,k}^{\alpha,\gamma,\kappa}(x,y)
S_{m,l}^{\alpha,\gamma,\kappa}(x,y)
w_{\alpha,\gamma,\kappa}(x,y)\,dx\,dy
=
H_{n,k}^{\alpha,\gamma,\kappa}\,
\delta_{n,m}\delta_{k,l},
\]
where
\[
b_{\alpha,\gamma,\kappa}
=
\left(
\int_{\triangle}
w_{\alpha,\gamma,\kappa}(x,y)\,dx\,dy
\right)^{-1},
\qquad
\alpha_j=2\alpha+2\kappa+4j+2,
\]
and
\begin{align*}
H_{n,k}^{\alpha,\gamma,\kappa}
&=
\frac{(\alpha_0+1)_{4k}}
{(\alpha_0+\gamma+2)_{4k}}
\frac{
(\alpha_k+1)_n(\gamma+1)_n(\alpha_k+\gamma+1)
}{
n!\,(\alpha_k+\gamma+1)_n(2n+\alpha_k+\gamma+1)
}
\\[5pt]
&\quad\times
\frac{
(\kappa+1)_k(\alpha+1)_k(\alpha+\kappa+1)
}{
k!\,(\alpha+\kappa+1)_k(2k+\alpha+\kappa+1)
}.
\end{align*}
For compactness, this formula is understood after cancelling the common
factors at exceptional parameter values.
\end{theorem}

\begin{proof}
Set
\[
u=x+y,
\qquad
t=\frac{(x-y)^2}{(x+y)^2}.
\]
Since $(x,y)\in\triangle$, we have $x-y>0$ and
\[
x=\frac{u(1+\sqrt{t})}{2},
\qquad
y=\frac{u(1-\sqrt{t})}{2},
\qquad
0<u<1,\quad 0<t<1.
\]
Moreover,
\[
xy=\frac{u^2}{4}(1-t),
\qquad
x-y=u\sqrt{t},
\qquad
dx\,dy=\frac{u}{4\sqrt{t}}\,du\,dt.
\]
Hence
\[
w_{\alpha,\gamma,\kappa}(x,y)\,dx\,dy
=
4^{-\alpha-1}
u^{2\alpha+2\kappa+2}(1-u)^\gamma
t^\kappa(1-t)^\alpha\,du\,dt .
\]
By \eqref{S_nk_agk},
\[
S_{n,k}^{\alpha,\gamma,\kappa}(x,y)
=
u^{2k}
P_n^{(\alpha_k,\gamma)}(1-2u)
P_k^{(\kappa,\alpha)}(1-2t).
\]
Thus the unnormalized inner product of
$S_{n,k}^{\alpha,\gamma,\kappa}$ and
$S_{m,l}^{\alpha,\gamma,\kappa}$ is
\[
\begin{aligned}
&4^{-\alpha-1}
\int_0^1\int_0^1
u^{2k+2l}
P_n^{(\alpha_k,\gamma)}(1-2u)
P_m^{(\alpha_l,\gamma)}(1-2u)
\\
&\qquad\qquad\qquad\times
P_k^{(\kappa,\alpha)}(1-2t)
P_l^{(\kappa,\alpha)}(1-2t)
u^{2\alpha+2\kappa+2}(1-u)^\gamma
t^\kappa(1-t)^\alpha\,dt\,du .
\end{aligned}
\]
The assumptions \(\alpha,\kappa>-1\) show that
\(t^\kappa(1-t)^\alpha\) is a Jacobi weight. Moreover,
\[
\alpha_k
=
2\alpha+2\kappa+4k+2
>
-1,
\qquad k\geq0,
\]
because \(\alpha+\kappa>-3/2\), while \(\gamma>-1\). Thus
\(u^{\alpha_k}(1-u)^\gamma\) is also a Jacobi weight.

If $k\ne l$, the integral in $t$ vanishes by the orthogonality of the
Jacobi polynomials. If $k=l$, then $\alpha_k=\alpha_l$, and the remaining
integral in $u$ is
\[
\int_0^1
P_n^{(\alpha_k,\gamma)}(1-2u)
P_m^{(\alpha_k,\gamma)}(1-2u)
u^{\alpha_k}(1-u)^\gamma\,du,
\]
which vanishes whenever $n\ne m$. Therefore the system is mutually
orthogonal.

It remains to compute $H_{n,k}^{\alpha,\gamma,\kappa}$. For
\(a,b>-1\), denote
\[
h_j^{(a,b)}
=
\int_0^1
\bigl(P_j^{(a,b)}(1-2s)\bigr)^2s^a(1-s)^b\,ds .
\]
Using \cite[(4.3.3)]{Szego75}, after the change of variables
\(x=1-2s\), we have
\[
\frac{h_j^{(a,b)}}{h_0^{(a,b)}}
=
\frac{
(a+1)_j(b+1)_j(a+b+1)
}{
j!\,(a+b+1)_j(2j+a+b+1)
},
\]
where the quotient is understood after cancelling the common factor
\(a+b+1\) before specializing the parameters. This convention also
covers \(a+b+1=0\), and the quotient equals \(1\) when \(j=0\).
For $n=m$ and $k=l$, the unnormalized inner product of
$S_{n,k}^{\alpha,\gamma,\kappa}$ with itself is
\[
4^{-\alpha-1}h_n^{(\alpha_k,\gamma)}h_k^{(\kappa,\alpha)},
\]
whereas
\[
\int_{\triangle}w_{\alpha,\gamma,\kappa}(x,y)\,dx\,dy
=
4^{-\alpha-1}h_0^{(\alpha_0,\gamma)}h_0^{(\kappa,\alpha)}.
\]
Consequently,
\[
H_{n,k}^{\alpha,\gamma,\kappa}
=
\frac{h_n^{(\alpha_k,\gamma)}h_k^{(\kappa,\alpha)}}
{h_0^{(\alpha_0,\gamma)}h_0^{(\kappa,\alpha)}}.
\]
Finally,
\[
\frac{h_0^{(\alpha_k,\gamma)}}{h_0^{(\alpha_0,\gamma)}}
=
\frac{(\alpha_0+1)_{4k}}{(\alpha_0+\gamma+2)_{4k}},
\]
since $\alpha_k=\alpha_0+4k$. Combining these identities gives the stated
formula for $H_{n,k}^{\alpha,\gamma,\kappa}$.
\end{proof}

We now record the leading term of the explicit family. The leading
coefficient of a Jacobi polynomial is
\[
P_j^{(a,b)}(\xi)
=
\frac{(j+a+b+1)_j}{2^j j!}\,\xi^j
+
\text{terms of lower degree}.
\]
Moreover,
\[
1-2\frac{(x-y)^2}{(x+y)^2}
=
-1+\frac{8xy}{(x+y)^2}.
\]
Thus, with respect to the order \(\preceq\), the leading contribution in
\eqref{S_nk_agk} is obtained by taking the leading term of each Jacobi
factor, and we have
\begin{align*}
S_{n,k}^{\alpha,\gamma,\kappa}(x,y)
&=
\frac{(-1)^n4^k
(n+2\alpha+2\kappa+4k+\gamma+3)_n
(k+\alpha+\kappa+1)_k}
{n!\,k!}
\\
&\qquad\times
(x+y)^n(xy)^k
+
\text{terms of symmetric degree lower than }(n,k).
\end{align*}
The standing parameter assumptions ensure that this leading coefficient
is nonzero.
Thus, if we set
\begin{equation*}%\label{s_nk_agk_explicit}
s_{n,k}^{\alpha,\gamma,\kappa}(x,y)
=
c_{n,k}^{\alpha,\gamma,\kappa}
S_{n,k}^{\alpha,\gamma,\kappa}(x,y),
\end{equation*}
where
\[
c_{n,k}^{\alpha,\gamma,\kappa}
=
\frac{(-1)^n n!\,k!}
{4^k
(n+2\alpha+2\kappa+4k+\gamma+3)_n
(k+\alpha+\kappa+1)_k},
\]
then
\[
s_{n,k}^{\alpha,\gamma,\kappa}(x,y)
=
(x+y)^n(xy)^k
+
\text{terms of symmetric degree lower than }(n,k).
\]

We now compare this normalized explicit family with the monic orthogonal
basis obtained by the Gram--Schmidt procedure.

\begin{proposition}%\label{prop:s_nk_equals_p_nk}
For $n,k\in\mathbb{N}_0$,
\[
s_{n,k}^{\alpha,\gamma,\kappa}(x,y)
=
p_{n,k}^{\alpha,\gamma,\kappa}(x,y).
\]
\end{proposition}

\begin{proof}
By Theorem~\ref{th:S_nk_orthogonal}, the family
$\{S_{n,k}^{\alpha,\gamma,\kappa}\}$ is mutually orthogonal with respect
to $w_{\alpha,\gamma,\kappa}$ on $\triangle$. Since
$s_{n,k}^{\alpha,\gamma,\kappa}$ is a nonzero scalar multiple of
$S_{n,k}^{\alpha,\gamma,\kappa}$, the family
$\{s_{n,k}^{\alpha,\gamma,\kappa}\}$ is mutually orthogonal as well.

Moreover, $s_{n,k}^{\alpha,\gamma,\kappa}$ has symmetric degree $(n,k)$
and leading term $(x+y)^n(xy)^k$. Hence its expansion in the orthogonal
basis obtained by Gram--Schmidt has the form
\[
s_{n,k}^{\alpha,\gamma,\kappa}
=
p_{n,k}^{\alpha,\gamma,\kappa}
+
\sum_{(m,l)\prec(n,k)}
a_{m,l}\,p_{m,l}^{\alpha,\gamma,\kappa}.
\]
Taking the inner product with $p_{m,l}^{\alpha,\gamma,\kappa}$, for
$(m,l)\prec(n,k)$, gives
\[
\resizebox{\textwidth}{!}{$ \displaystyle a_{m,l}\,
\int_{\triangle}
\bigl(p_{m,l}^{\alpha,\gamma,\kappa}(x,y)\bigr)^2
w_{\alpha,\gamma,\kappa}(x,y)\,dx\,dy
=
\int_{\triangle}
s_{n,k}^{\alpha,\gamma,\kappa}(x,y)
p_{m,l}^{\alpha,\gamma,\kappa}(x,y)
w_{\alpha,\gamma,\kappa}(x,y)\,dx\,dy .$}
\]
The integral on the right-hand side is zero by
Theorem~\ref{th:S_nk_orthogonal}, since
$s_{n,k}^{\alpha,\gamma,\kappa}$ is a scalar multiple of
$S_{n,k}^{\alpha,\gamma,\kappa}$ and
$p_{m,l}^{\alpha,\gamma,\kappa}$ is a linear combination of
$S_{r,s}^{\alpha,\gamma,\kappa}$ with $(r,s)\preceq(m,l)\prec(n,k)$.
Since $\alpha,\gamma,\kappa>-1$ and $\alpha+\kappa>-\frac32$, the norm multiplying
$a_{m,l}$ is finite and strictly positive. Therefore \(a_{m,l}=0\).
Thus
\[
s_{n,k}^{\alpha,\gamma,\kappa}
=
p_{n,k}^{\alpha,\gamma,\kappa},
\]
as claimed.
\end{proof}

%------------------------------------------------------------------------------
\section{A change of variables}\label{sec:uv}
%------------------------------------------------------------------------------

The change of variables introduced in this section rewrites the symmetric
problem on the chamber as an ordinary bivariate orthogonality problem on a
planar domain. This formulation will be used to express the differential
operators in coordinates adapted to the symmetric structure.

Consider the change of variables
\[
u=x+y,\qquad v=xy.
\]
This map sends the interior of \(\triangle\) bijectively onto the region
\(\Omega\) shown in Figure~\ref{fig:Omega-region}, where
\[
\Omega
=
\left\{(u,v)\in\mathbb{R}^2:\ 0< u< 1,\ 0< v< \frac{u^2}{4}\right\}.
\]

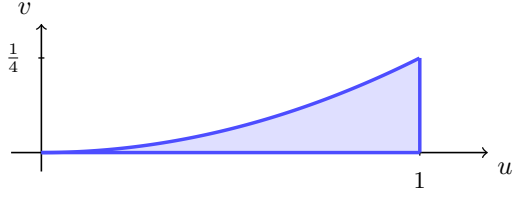
\begin{figure}[t]
	\centering
	\begin{tikzpicture}[scale=5]
		
		% Ventana
		\def\umin{-0.08} \def\umax{1.18}
		\def\vmin{-0.05} \def\vmax{0.34}
		
		% Ejes
		\draw[->,black,line width=0.6pt] (\umin,0) -- (\umax,0) node[below right] {$u$};
		\draw[->,black,line width=0.6pt] (0,\vmin) -- (0,\vmax) node[above left] {$v$};
		
		% Marcas y etiquetas
		\foreach \x in {1}{
			\draw[black,line width=0.5pt] (\x,0.008)--(\x,-0.008);
			\node[below=4pt] at (\x,0) {\small $\x$};
		}
		\foreach \y/\ylabel in {0.25/{\frac14}}{
			\draw[black,line width=0.5pt] (0.008,\y)--(-0.008,\y);
			\node[left=4pt] at (0,\y) {\small $\ylabel$};
		}
		
		% Región Ω = {(u,v): 0 <= u <= 1, 0 <= v <= u^2/4}
		\begin{scope}
			\clip (\umin,\vmin) rectangle (\umax,\vmax);
			\fill[blue!35,fill opacity=0.35]
			(0,0)
			-- plot[domain=0:1,samples=250] ({\x},{(\x*\x)/4})
			-- (1,0)
			-- cycle;
		\end{scope}
		
		% Bordes
		\draw[line width=1.3pt,blue!70,domain=0:1,samples=250]
		plot ({\x},{(\x*\x)/4});
		\draw[line width=1.3pt,blue!70] (0,0) -- (1,0);
		\draw[line width=1.3pt,blue!70] (1,0) -- (1,0.25);
		
		% Puntos destacados
		%\fill[blue!70] (0,0) circle (0.6pt);
		%\fill[blue!70] (1,0) circle (0.6pt);
		%\fill[blue!70] (1,0.25) circle (0.6pt);
		
	\end{tikzpicture}

	\caption{ $\Omega=\left\{(u,v)\in\mathbb{R}^2:\ 0< u< 1,\ 0< v< \dfrac{u^2}{4}\right\}$.}
	\label{fig:Omega-region}
\end{figure}
Since the Jacobian of the change of variables is $|x-y|
=
\sqrt{u^2-4v}$, 
we have
\[
\int_{\Omega} f(u,v)\,
\mathcal{W}_{\alpha,\gamma,\kappa}(u,v)\,du\,dv
=
\int_{\triangle}
f(x+y,xy)\,
w_{\alpha,\gamma,\kappa}(x,y)\,dx\,dy,
\]
where
\[
\mathcal{W}_{\alpha,\gamma,\kappa}(u,v)
=
v^\alpha(1-u)^\gamma(u^2-4v)^\kappa,
\qquad
(u,v)\in\Omega.
\]

Every polynomial $q(u,v)$ can be written as a finite linear combination of
the monomials
\[
\{ u^n v^k,\quad n,k\in\mathbb{N}_0\}.
\]
We order these monomials using the order $\preceq$.
Thus, if
\[
q(u,v)
=
c_{n,k}u^n v^k
+
\sum_{(m,l)\prec(n,k)}
c_{m,l}u^m v^l,
\qquad c_{n,k}\ne0,
\]
then we say that $q(u,v)$ has degree $(n,k)$. Notice that the weighted
degree associated with the leading monomial $u^nv^k$ is $n+2k$.

Finally, if $q(u,v)$ has degree $(n,k)$, then $q(x+y,xy)$ defines a symmetric polynomial in $x$ and $y$ of symmetric degree $(n,k)$.
Moreover, the coefficient of $(x+y)^n(xy)^k$ in $q(x+y,xy)$ coincides
with the coefficient of $u^nv^k$ in $q(u,v)$.

Hence any symmetric polynomial in $x$ and $y$ can be written uniquely in
terms of $u=x+y$ and $v=xy$. For the orthogonal basis $\{p_{n,k}^{\alpha,\gamma,\kappa}(x,y):\ n,k\in\mathbb{N}_0\}$, define
\[
q_{n,k}^{\alpha,\gamma,\kappa}(x+y,xy)
:=
p_{n,k}^{\alpha,\gamma,\kappa}(x,y).
\]
Then
\[
\resizebox{\textwidth}{!}{$\displaystyle \int_{\Omega}
q_{n,k}^{\alpha,\gamma,\kappa}(u,v)
q_{m,l}^{\alpha,\gamma,\kappa}(u,v)
\mathcal{W}_{\alpha,\gamma,\kappa}(u,v)dudv
=
\int_{\triangle}
p_{n,k}^{\alpha,\gamma,\kappa}(x,y)
p_{m,l}^{\alpha,\gamma,\kappa}(x,y)
w_{\alpha,\gamma,\kappa}(x,y)dxdy,$}
\]
and therefore we say that
\[
\{q_{n,k}^{\alpha,\gamma,\kappa}(u,v):\ n,k\in\mathbb{N}_0\}
\]
is an orthogonal polynomial system with respect to
$\mathcal{W}_{\alpha,\gamma,\kappa}$ on $\Omega$.

Under the change of variables,
\[
\partial_x=\partial_u+y\,\partial_v,
\qquad
\partial_y=\partial_u+x\,\partial_v.
\]
Thus
\begin{align*}
\partial_x^2
&=
\partial_u^2+2y\,\partial_u\partial_v+y^2\partial_v^2,\\[5pt]
\partial_x\partial_y
&=
\partial_u^2+u\,\partial_u\partial_v+v\,\partial_v^2+\partial_v,\\[5pt]
\partial_y^2
&=
\partial_u^2+2x\,\partial_u\partial_v+x^2\partial_v^2.
\end{align*}
Consequently, in $(u,v)$-coordinates, the operator
$\mathcal{D}_{1}^{\alpha,\gamma,\kappa}$ becomes
\[
\begin{aligned}
\mathcal{D}_{1}^{\alpha,\gamma,\kappa}
&=
u(1-u)\partial_u^2
+
4v(1-u)\partial_u\partial_v
+
v(u-4v)\partial_v^2
\\
&\quad
+
\bigl[2\alpha+2\kappa+3
-
(2\alpha+2\kappa+\gamma+4)u\bigr]\partial_u
\\
&\quad
+
\bigl[(\alpha+1)u
-
2(2\alpha+2\kappa+\gamma+5)v\bigr]\partial_v .
\end{aligned}
\]
Therefore
\[
\mathcal{D}_{1}^{\alpha,\gamma,\kappa}
\big[q_{n,k}^{\alpha,\gamma,\kappa}(u,v)\big]
=
\lambda_{n,k}^{\alpha,\gamma,\kappa}
q_{n,k}^{\alpha,\gamma,\kappa}(u,v),
\]
where
\[
\lambda_{n,k}^{\alpha,\gamma,\kappa}
=
-(n+2k)(n+2k+2\alpha+2\kappa+\gamma+3).
\]

The operators $D_{-}$ and $D_{+}^{\alpha,\kappa}$ take the form
\[
D_{-}=\partial_v
\]
and
\[
D_{+}^{\alpha,\kappa}
=
-v(u^2-4v)\partial_v
-(\alpha+1)u^2
+
4(\alpha+\kappa+2)v.
\]
Hence Theorem~\ref{th:ladderD_agk} gives
\[
D_{-}
\big[q_{n,k}^{\alpha,\gamma,\kappa}(u,v)\big]
=
k\,q_{n,k-1}^{\alpha+1,\gamma,\kappa+1}(u,v),
\]
with the usual convention that the right-hand side is zero when $k=0$,
and
\[
D_{+}^{\alpha,\kappa}
\big[q_{n,k}^{\alpha+1,\gamma,\kappa+1}(u,v)\big]
=
4(k+\alpha+\kappa+2)\,
q_{n,k+1}^{\alpha,\gamma,\kappa}(u,v).
\]

Finally, the operator
\[
\mathcal{D}_{2}^{\alpha,\kappa}
=
D_{+}^{\alpha,\kappa}D_{-}
\]
is given in $(u,v)$-coordinates by
\[
\mathcal D_2^{\alpha,\kappa}
=
-v(u^2-4v)\partial_v^2
+
\left[-(\alpha+1)u^2+4(\alpha+\kappa+2)v\right]\partial_v,
\]
and satisfies
\[
\mathcal{D}_{2}^{\alpha,\kappa}
\big[q_{n,k}^{\alpha,\gamma,\kappa}(u,v)\big]
=
4k(k+\alpha+\kappa+1)
q_{n,k}^{\alpha,\gamma,\kappa}(u,v).
\]

Thus the variables \(u=x+y\) and \(v=xy\) provide the coordinate system in
which the orthogonality, the spectral equations, and the operator algebra
can be studied as a two-variable polynomial system.

%------------------------------------------------------------------------------
\section{The algebra of differential operators}\label{sec:algebra}
%------------------------------------------------------------------------------

We have established that the transformed polynomials
\(\{q_{n,k}^{\alpha,\gamma,\kappa}(u,v)\}\) are common eigenfunctions of
the second-order operators
\(\mathcal D_1^{\alpha,\gamma,\kappa}\) and
\(\mathcal D_2^{\alpha,\kappa}\). Let \(\mathfrak D_{\alpha,\gamma,\kappa}\) denote the algebra of linear
differential operators with real polynomial coefficients in \(u\) and \(v\)
such that every \(q_{n,k}^{\alpha,\gamma,\kappa}\) is an eigenfunction.
We now determine this algebra.

From Theorem~\ref{th:D1_eigen} and Corollary~\ref{coro:D2_agk}, the
eigenvalues of the two basic operators are
\begin{equation}\label{eq:eigD2_agk}
\begin{aligned}
\mathcal{D}_{1}^{\alpha,\gamma,\kappa}
\big[q_{n,k}^{\alpha,\gamma,\kappa}(u,v)\big]
&=
\lambda_{n,k}^{\alpha,\gamma,\kappa}
q_{n,k}^{\alpha,\gamma,\kappa}(u,v),
%\label{eq:eigD1_agk}
\\[4pt]
\mathcal{D}_{2}^{\alpha,\kappa}
\big[q_{n,k}^{\alpha,\gamma,\kappa}(u,v)\big]
&=
\mu_{n,k}^{\alpha,\kappa}
q_{n,k}^{\alpha,\gamma,\kappa}(u,v),
\end{aligned}
\end{equation}
where
\[
\lambda_{n,k}^{\alpha,\gamma,\kappa}
=
-(n+2k)(n+2k+2\alpha+2\kappa+\gamma+3),
\]
and
\[
\mu_{n,k}^{\alpha,\kappa}
=
4k(k+\alpha+\kappa+1).
\]
Thus the first eigenvalue depends on the quantity \(n+2k\), whereas the
second one depends only on \(k\). 

Let \(\mathcal{D}\in\mathfrak{D}_{\alpha,\gamma,\kappa}\). We write $\mathcal D$ as a finite sum
\[
\mathcal D
=
\sum_{r,s\geq0}
a_{r,s}(u,v)\,\partial_u^r\partial_v^s,
\]
for suitable polynomials \(a_{r,s}(u,v)\). Since \(\mathcal{D}\) belongs to
\(\mathfrak{D}_{\alpha,\gamma,\kappa}\), there exists a function
\(\Lambda_{\mathcal{D}}(n,k)\) such that
\[
\mathcal{D}
\big[q_{n,k}^{\alpha,\gamma,\kappa}(u,v)\big]
=
\Lambda_{\mathcal{D}}(n,k)\,
q_{n,k}^{\alpha,\gamma,\kappa}(u,v).
\]
It follows from \cite[Lemma~7.1]{AMP25} that the coefficient
$a_{r,s}(u,v)$ cannot have degree larger than \((r,s)\). The same result
also implies that an operator in \(\mathfrak{D}_{\alpha,\gamma,\kappa}\)
is completely determined by its eigenvalue sequence on
\(\{q_{n,k}^{\alpha,\gamma,\kappa}\}\). Hence the algebra
\(\mathfrak{D}_{\alpha,\gamma,\kappa}\) is commutative; see
\cite[Corollary~7.2]{AMP25}.

We return to the separating variable introduced in
Section~\ref{sec:explicit-jacobi},
\[
t=1-\frac{4v}{u^2}
=\frac{(x-y)^2}{(x+y)^2}.
\]
The following proposition describes the structure of the highest-order terms.

\begin{proposition}\label{lem:highest_order_generated}
Let \(\mathcal D\in\mathfrak D_{\alpha,\gamma,\kappa}\) be a nonzero
differential operator. Then \(\mathcal D\) has even order \(2m\).
Moreover, in the variables \((u,t)\), the operator has an expansion of the
form
\begin{equation}\label{eq:highest-order-expansion-generated}
\mathcal D
=
\sum_{\ell=0}^{m}
c_\ell\,
\left(\mathcal D_{1}^{\alpha,\gamma,\kappa}\right)^{m-\ell}
\left(\mathcal D_{2}^{\alpha,\kappa}\right)^{\ell}
+
\textup{terms of order lower than }2m
\end{equation}
with constant coefficients \(c_\ell\).
\end{proposition}

\begin{proof}
Let \(N\) be the order of \(\mathcal D\), and write
\[
\mathcal D
=
\sum_{r+s=N}
b_{r,s}(u,t)\,
\partial_u^{\,r}\partial_t^{\,s}
+
\textup{terms of lower order}.
\]
Since the operators in \(\mathfrak D_{\alpha,\gamma,\kappa}\) have the
same polynomial eigenbasis, \(\mathcal D\) commutes with
\(\mathcal D_1^{\alpha,\gamma,\kappa}\) and
\(\mathcal D_2^{\alpha,\kappa}\). We first determine the terms of highest
order in \(\mathcal D\) from these two commutation relations.

In the variables \((u,t)\), the order-two parts of the two basic
operators are
\begin{align*}
\mathcal D_2^{\alpha,\kappa}
&=
-4t(1-t)\partial_t^2
+
\textup{terms of lower order},\\
\mathcal D_1^{\alpha,\gamma,\kappa}
&=
u(1-u)\partial_u^2
+
\frac{4t(1-t)}{u}\partial_t^2
+
\textup{terms of lower order}.
\end{align*}
Comparing the terms of order \(N+1\) in
\[
[\mathcal D,\mathcal D_2^{\alpha,\kappa}]=0
\]
gives, for \(r+s=N\),
\[
s(1-2t)b_{r,s}(u,t)
-
2t(1-t)\partial_t b_{r,s}(u,t)
=
0.
\]
Hence
\[
b_{r,s}(u,t)
=
d_{r,s}(u)[t(1-t)]^{s/2}.
\]
Since the coefficients of \(\mathcal D\), written in the variables
\((u,t)\), are rational functions of \(u\) and \(t\), the terms with
\(s\) odd must vanish. Thus the terms of order \(N\) in \(\mathcal D\)
can be written as
\begin{equation}\label{eq:D-highest-order-N}
\sum_{j=0}^{\lfloor N/2\rfloor}
d_j(u)
\left(-4t(1-t)\partial_t^2\right)^j
\partial_u^{\,N-2j}.
\end{equation}

The second commutation relation determines the dependence of the
coefficients \(d_j\) on \(u\). Comparing the terms of order \(N+1\) in
\[
[\mathcal D,\mathcal D_1^{\alpha,\gamma,\kappa}]=0
\]
gives
\begin{equation}\label{eq:highest-order-recurrence-compact-N}
\left(
\frac{d_j(u)}
{[u(1-u)]^{\frac{N}{2}-j}}
\right)'
=
\left(\frac{N}{2}-j+1\right)u^{-2}
\frac{d_{j-1}(u)}
{[u(1-u)]^{\frac{N}{2}-j+1}},
\quad
0\leq j\leq \left\lfloor N/2\right\rfloor,
\end{equation}
with the convention \(d_{-1}=0\).

The recurrence also forces \(N\) to be even. For \(j=0\),
\eqref{eq:highest-order-recurrence-compact-N} gives
\[
d_0(u)=C_0[u(1-u)]^{N/2}.
\]
Suppose, to the contrary, that \(N\) is odd. Since \(d_0(u)\) is
rational, the preceding identity forces \(C_0=0\), and hence
\(d_0(u)=0\). Assume inductively that
\[
d_0(u)=\cdots=d_{k-1}(u)=0.
\]
Then the right-hand side of
\eqref{eq:highest-order-recurrence-compact-N} vanishes for \(j=k\), and
therefore
\[
d_k(u)=C_k[u(1-u)]^{N/2-k}.
\]
Since \(N/2-k\) is not an integer, the rationality of \(d_k(u)\) forces
\(C_k=0\). Hence all terms of order \(N\) vanish, contradicting the
definition of \(N\). Thus \(N=2m\).

Set
\[
F_j(u)
=
\frac{d_j(u)}{[u(1-u)]^{m-j}},
\qquad 0\leq j\leq m.
\]
Then \eqref{eq:highest-order-recurrence-compact-N} becomes
\begin{equation}\label{eq:Fj_recurrence}
F_j'(u)
=
(m-j+1)u^{-2}F_{j-1}(u),
\qquad F_{-1}=0.
\end{equation}

It remains to express the functions \(F_j\) as a constant linear
combination of the coefficient families arising from the terms of order
\(2m\) in
\[
\left(\mathcal D_1^{\alpha,\gamma,\kappa}\right)^{m-q}
\left(\mathcal D_2^{\alpha,\kappa}\right)^q,
\qquad 0\leq q\leq m.
\]
For fixed \(q\), these terms can be written as
\begin{equation}\label{eq:Dq-highest-order-2m}
\sum_{j=0}^{m}
p_{j,q}(u)
[u(1-u)]^{m-j}
\left(-4t(1-t)\partial_t^2\right)^j
\partial_u^{\,2m-2j},
\end{equation}
where
\[
p_{j,q}(u)
=
\begin{cases}
0, & j<q,\\[4pt]
(-1)^{j-q}
\binom{m-q}{j-q}
u^{-(j-q)}, & j\ge q.
\end{cases}
\]
The functions \(p_{j,q}\) satisfy
\begin{equation}\label{eq:pjq-recurrence}
p_{j,q}'(u)
=
(m-j+1)u^{-2}p_{j-1,q}(u),
\end{equation}
with the convention \(p_{-1,q}=0\). Indeed, if \(j<q\), then both sides
of \eqref{eq:pjq-recurrence} vanish. If \(j=q\), then \(p_{q,q}=1\) and
\(p_{q-1,q}=0\), so the identity also holds. Finally, if \(j>q\), then
\[
\begin{aligned}
p_{j,q}'(u)
&=
(-1)^{j-q+1}(j-q)
\binom{m-q}{j-q}
u^{-(j-q+1)}
\\
&=
(-1)^{j-q-1}(m-j+1)
\binom{m-q}{j-q-1}
u^{-(j-q+1)}
\\
&=
(m-j+1)u^{-2}p_{j-1,q}(u),
\end{aligned}
\]
where the middle equality uses
\[
(j-q)\binom{m-q}{j-q}
=
(m-j+1)\binom{m-q}{j-q-1}.
\]
Thus, for each \(q\), the functions \(p_{j,q}\) satisfy the same
recurrence as the functions \(F_j\), with \(p_{q,q}=1\) and
\(p_{j,q}=0\) for \(j<q\).

By \eqref{eq:D-highest-order-N}, with \(N=2m\), and
\eqref{eq:Dq-highest-order-2m}, equality of the terms of order \(2m\) is
equivalent to
\[
d_j(u)
=
[u(1-u)]^{m-j}
\sum_{q=0}^{j}c_qp_{j,q}(u),
\qquad 0\le j\le m,
\]
where the sum terminates at \(q=j\) because \(p_{j,q}=0\) for \(q>j\).
In terms of \(F_j\), this becomes the triangular system
\[
F_j(u)
=
\sum_{q=0}^{j}c_qp_{j,q}(u),
\qquad 0\le j\le m.
\]

We solve this system by induction on \(j\). Since \(F_0'=0\) and
\(p_{0,0}=1\), set
\[
c_0=F_0.
\]
Assume that \(c_0,\ldots,c_{j-1}\) have been chosen so that
\[
F_k(u)
=
\sum_{q=0}^{k}c_qp_{k,q}(u),
\qquad 0\le k<j.
\]
Define
\[
H_j(u)
=
F_j(u)
-
\sum_{q=0}^{j-1}c_qp_{j,q}(u).
\]
Applying \eqref{eq:Fj_recurrence} and \eqref{eq:pjq-recurrence} gives
\[
H_j'(u)
=
(m-j+1)u^{-2}
\left(
F_{j-1}(u)
-
\sum_{q=0}^{j-1}c_qp_{j-1,q}(u)
\right).
\]
The expression in parentheses vanishes by the induction hypothesis.
Hence \(H_j\) is constant, and we set
\[
c_j=H_j.
\]
Since \(p_{j,j}=1\), this gives the required identity at index \(j\) and
completes the induction.

Multiplying the resulting identities by \([u(1-u)]^{m-j}\), the
coefficients of the terms of order \(2m\) agree for every
\(0\leq j\leq m\). Therefore
\[
\mathcal D
-
\sum_{q=0}^{m}
c_q
\left(\mathcal D_1^{\alpha,\gamma,\kappa}\right)^{m-q}
\left(\mathcal D_2^{\alpha,\kappa}\right)^q
\]
has order lower than \(2m\).
\end{proof}

We now identify the generators of the algebra.

\begin{theorem}\label{thm:algebra_agk}
Any linear partial differential operator with real polynomial coefficients in
$u$ and $v$ that admits the polynomials
\[
\{q_{n,k}^{\alpha,\gamma,\kappa}(u,v):\ n,k\in\mathbb{N}_0\}
\]
as eigenfunctions can be written uniquely as a polynomial in
$\mathcal{D}_{1}^{\alpha,\gamma,\kappa}$ and
$\mathcal{D}_{2}^{\alpha,\kappa}$.
\end{theorem}

\begin{proof}
We first prove existence. Suppose that there exists a nonzero operator in
\(\mathfrak D_{\alpha,\gamma,\kappa}\) that cannot be written as a
polynomial in \(\mathcal D_1^{\alpha,\gamma,\kappa}\) and
\(\mathcal D_2^{\alpha,\kappa}\), and choose such an operator
\(\mathcal D\) of minimal order.

By Proposition~\ref{lem:highest_order_generated}, the order of
\(\mathcal D\) is \(2m\), and there exist constants
\(c_0,\ldots,c_m\) such that
\[
\widetilde{\mathcal D}
=
\mathcal D
-
\sum_{\ell=0}^{m}
c_\ell
\left(\mathcal D_1^{\alpha,\gamma,\kappa}\right)^{m-\ell}
\left(\mathcal D_2^{\alpha,\kappa}\right)^\ell
\]
has order lower than \(2m\). Since
\(\mathcal D_1^{\alpha,\gamma,\kappa}\) and
\(\mathcal D_2^{\alpha,\kappa}\) have the polynomials
\(q_{n,k}^{\alpha,\gamma,\kappa}\) as common eigenfunctions,
\(\widetilde{\mathcal D}\) also admits all these polynomials as
eigenfunctions. Moreover, \(\widetilde{\mathcal D}\) has polynomial
coefficients.

If \(\widetilde{\mathcal D}\neq0\), its order is lower than that of
\(\mathcal D\), so the minimality of \(\mathcal D\) implies that
\(\widetilde{\mathcal D}\) is a polynomial in
\(\mathcal D_1^{\alpha,\gamma,\kappa}\) and
\(\mathcal D_2^{\alpha,\kappa}\). The same conclusion is immediate if
\(\widetilde{\mathcal D}=0\). It follows in either case that
\(\mathcal D\) is a polynomial in these two operators, which contradicts
the choice of \(\mathcal D\).

We now prove uniqueness. It is enough to show that
\[
P\left(
\mathcal D_1^{\alpha,\gamma,\kappa},
\mathcal D_2^{\alpha,\kappa}
\right)
\neq0
\]
for every nonzero polynomial \(P(X,Y)\). Let \(r\) be the total degree of \(P\). If \(r=0\), then \(P\) is a
nonzero constant, and therefore
\[
P\left(
\mathcal D_1^{\alpha,\gamma,\kappa},
\mathcal D_2^{\alpha,\kappa}
\right)
\neq0.
\]
We may therefore assume that \(r\geq1\). Denote by
\[
P_r(X,Y)
=
\sum_{q=0}^{r}a_qX^{r-q}Y^q
\]
the homogeneous part of \(P\) of degree \(r\). 

Since \(P-P_r\) has degree at most \(r-1\) and both
\(\mathcal D_1^{\alpha,\gamma,\kappa}\) and
\(\mathcal D_2^{\alpha,\kappa}\) have order two, the differential
operator
\[
(P-P_r)\left(
\mathcal D_1^{\alpha,\gamma,\kappa},
\mathcal D_2^{\alpha,\kappa}
\right)
\]
has order at most \(2r-2\). Hence the terms of order \(2r\) in
\[
P\left(
\mathcal D_1^{\alpha,\gamma,\kappa},
\mathcal D_2^{\alpha,\kappa}
\right)
\]
coincide with the terms of order \(2r\) in
\[
P_r\left(
\mathcal D_1^{\alpha,\gamma,\kappa},
\mathcal D_2^{\alpha,\kappa}
\right).
\]

It remains to prove that the latter terms do not vanish. Let \(q_0\) be
the least index such that \(a_{q_0}\neq0\). For each
\(0\leq q\leq r\), equation \eqref{eq:Dq-highest-order-2m}, with
\(m=r\), gives the terms of order \(2r\) in
\[
\left(\mathcal D_1^{\alpha,\gamma,\kappa}\right)^{r-q}
\left(\mathcal D_2^{\alpha,\kappa}\right)^q
\]
as
\[
\sum_{j=0}^{r}
p_{j,q}(u)[u(1-u)]^{r-j}
\left(-4t(1-t)\partial_t^2\right)^j
\partial_u^{\,2r-2j}.
\]
Since \(p_{j,q}=0\) for \(j<q\), the first possibly nonzero term in this
sum corresponds to \(j=q\).

Consider the coefficient of
\[
\left(-4t(1-t)\partial_t^2\right)^{q_0}
\partial_u^{\,2r-2q_0}
\]
in the terms of order \(2r\) of
\[
P_r\left(
\mathcal D_1^{\alpha,\gamma,\kappa},
\mathcal D_2^{\alpha,\kappa}
\right).
\]
The summands with \(q<q_0\) do not contribute because \(a_q=0\), whereas
the summands with \(q>q_0\) do not contribute because
\(p_{q_0,q}=0\). Thus only the summand with \(q=q_0\) contributes, and
its coefficient is
\[
a_{q_0}p_{q_0,q_0}(u)[u(1-u)]^{r-q_0}
=
a_{q_0}[u(1-u)]^{r-q_0},
\]
since \(p_{q_0,q_0}=1\). This coefficient is not identically zero.
Therefore the part of order \(2r\) of
\[
P_r\left(
\mathcal D_1^{\alpha,\gamma,\kappa},
\mathcal D_2^{\alpha,\kappa}
\right)
\]
is nonzero.

The evaluation of \(P-P_r\) has order at most \(2r-2\), so it has no
terms of order \(2r\). Consequently, the same nonzero term occurs in
the part of order \(2r\) of
\[
P\left(
\mathcal D_1^{\alpha,\gamma,\kappa},
\mathcal D_2^{\alpha,\kappa}
\right).
\]
Thus this operator cannot vanish. Hence there is no nonzero polynomial
relation between \(\mathcal D_1^{\alpha,\gamma,\kappa}\) and
\(\mathcal D_2^{\alpha,\kappa}\), and the polynomial representation is
unique.
\end{proof}

\begin{corollary}%\label{cor:algebra_polynomial_ring_agk}
The algebra
$\mathfrak{D}_{\alpha,\gamma,\kappa}$ is isomorphic to the real polynomial ring in two variables.
\end{corollary}

\section{Concluding remarks}

We have developed the orthogonal and differential structure associated with
the symmetric Jacobi-type weight
\[
w_{\alpha,\gamma,\kappa}(x,y)
=
(xy)^\alpha(1-x-y)^\gamma |x-y|^{2\kappa+1}
\]
on a chamber of the simplex. The construction gives a monic symmetric
orthogonal basis, an explicit representation in terms of one-variable Jacobi
polynomials, and closed formulas for its squared norms. The basis is also a
common eigenbasis of two second-order differential operators:
\(\mathcal D_1^{\alpha,\gamma,\kappa}\) and the operator
\(\mathcal D_2^{\alpha,\kappa}\) obtained from the lowering and raising
operators.

The main algebraic result gives a complete description of the differential
operators associated with this eigenbasis. After passing to the elementary
symmetric variables \(u=x+y\) and \(v=xy\), every linear partial differential
operator with real polynomial coefficients having the transformed polynomials as
eigenfunctions is a polynomial in
\(\mathcal D_1^{\alpha,\gamma,\kappa}\) and
\(\mathcal D_2^{\alpha,\kappa}\). This representation is unique, so the two
operators are algebraically independent and the resulting differential-operator
algebra is isomorphic to the real polynomial ring in two variables. In particular,
the two second-order operators generate the complete algebra, and every nonzero
operator in it has even order.

\appendix

\section{Boundary estimates}\label{app:boundary-estimates}

The following lemmas collect the boundary estimates used in the Green
identities for
\(\mathcal D_1^{\alpha,\gamma,\kappa}\) and for the pair
\(D_-\), \(D_+^{\alpha,\kappa}\).

We assume that $\alpha,\gamma,\kappa>-1$ and $\alpha+\kappa>-\frac32$. For \(\varepsilon>0\), set
\[
\triangle_\varepsilon
=
\{(x,y)\in\triangle:
  y>\varepsilon,\ 
  z>\varepsilon,\ 
  x-y>\varepsilon\},
\qquad
z=1-x-y.
\]
Its boundary consists of the three sides
\[
\Gamma_\varepsilon^{(y)}
=
\{y=\varepsilon\},
\qquad
\Gamma_\varepsilon^{(z)}
=
\{z=\varepsilon\},
\qquad
\Gamma_\varepsilon^{(d)}
=
\{x-y=\varepsilon\},
\]
restricted by the remaining inequalities defining
\(\triangle_\varepsilon\).

For a symmetric polynomial \(f\), write
\[
A_f
=
x(1-x)\partial_xf-xy\,\partial_yf,
\qquad
B_f
=
y(1-y)\partial_yf-xy\,\partial_xf.
\]

\begin{lemma}\label{lem:D1-boundary-estimates}
Let \(f\) and \(g\) be symmetric polynomials. Then
\[
\int_{\partial\triangle_\varepsilon}
g\,w_{\alpha,\gamma,\kappa}
\left(A_f\,dy-B_f\,dx\right)
\longrightarrow0
\qquad
\text{as }\varepsilon\to0^+.
\]
\end{lemma}

\begin{proof}
We estimate separately the contributions from the three sides of
\(\partial\triangle_\varepsilon\).

On \(\Gamma_\varepsilon^{(y)}\), the boundary contribution is
determined by \(B_f\). Since \(f\), \(g\), and their first derivatives
are bounded on the closed triangle, there exists a constant \(C>0\),
independent of \(\varepsilon\), such that
\[
\left|I_\varepsilon^{(y)}\right|
\leq
C\varepsilon^{\alpha+1}
\int_{2\varepsilon}^{1-2\varepsilon}
x^\alpha
(1-x-\varepsilon)^\gamma
(x-\varepsilon)^{2\kappa+1}\,dx.
\]

Fix \(0<\delta<1/4\) and assume that
\(0<\varepsilon<\delta/2\). We decompose
\(I_\varepsilon^{(y)}
=I_{\varepsilon,0}^{(y)}+I_{\varepsilon,1}^{(y)}\), where the two terms
are the contributions from \([2\varepsilon,\delta]\) and
\([\delta,1-2\varepsilon]\), respectively. On \([2\varepsilon,\delta]\),
\[
\frac{x}{2}\leq x-\varepsilon\leq x,
\qquad
1-x-\varepsilon
\geq
1-\frac{3\delta}{2}>0,
\]
and hence
\[
\left|I_{\varepsilon,0}^{(y)}\right|
\leq
C\varepsilon^{\alpha+1}
\int_{2\varepsilon}^{\delta}
x^{\alpha+2\kappa+1}\,dx.
\]
Set
\[
\rho=\alpha+2\kappa+1.
\]
If \(\rho>-1\), the integral is uniformly bounded. If \(\rho=-1\),
it is bounded by a constant multiple of
\(\log(1/\varepsilon)\). If \(\rho<-1\), it is bounded by a constant
multiple of \(\varepsilon^{\rho+1}\), so
\[
\left|I_{\varepsilon,0}^{(y)}\right|
\leq
C\varepsilon^{2\alpha+2\kappa+3}
\longrightarrow0
\]
because \(\alpha+\kappa>-3/2\). Thus
\(I_{\varepsilon,0}^{(y)}\to0\) in every case.

On \([\delta,1-2\varepsilon]\), the factors \(x\) and
\(x-\varepsilon\) are uniformly separated from zero, while
\(\gamma>-1\) gives
\[
\left|I_{\varepsilon,1}^{(y)}\right|
\leq
C\varepsilon^{\alpha+1}
\longrightarrow0.
\]
Consequently,
\[
I_\varepsilon^{(y)}\longrightarrow0.
\]

On \(\Gamma_\varepsilon^{(z)}\), the boundary contribution is
determined by \(A_f+B_f\). Since
\[
A_f+B_f
=
z\bigl(x\partial_xf+y\partial_yf\bigr),
\]
we have
\[
\left|I_\varepsilon^{(z)}\right|
\leq
C\varepsilon^{\gamma+1}
\int_{1/2}^{1-2\varepsilon}
\bigl[x(1-\varepsilon-x)\bigr]^\alpha
(2x-1+\varepsilon)^{2\kappa+1}\,dx.
\]
Introduce
\[
r
=
\frac{2x-1+\varepsilon}{1-\varepsilon}.
\]
Then
\[
x=\frac{1-\varepsilon}{2}(1+r),
\qquad
1-\varepsilon-x=\frac{1-\varepsilon}{2}(1-r),
\qquad
dx=\frac{1-\varepsilon}{2}\,dr,
\]
and the integral becomes
\[
2^{-2\alpha-1}
(1-\varepsilon)^{2\alpha+2\kappa+2}
\int_{\frac{\varepsilon}{1-\varepsilon}}^{
      \frac{1-3\varepsilon}{1-\varepsilon}}
(1-r^2)^\alpha r^{2\kappa+1}\,dr.
\]
For \(0<\varepsilon<1/4\), this expression is uniformly bounded by
a constant multiple of
\[
\int_0^1
(1-r^2)^\alpha r^{2\kappa+1}\,dr,
\]
which is finite because \(\alpha,\kappa>-1\). Therefore
\[
\left|I_\varepsilon^{(z)}\right|
\leq
C\varepsilon^{\gamma+1}
\longrightarrow0.
\]

Finally, on \(\Gamma_\varepsilon^{(d)}\), symmetry of \(f\) implies
that there exists a polynomial \(C_f(x,y)\) such that
\[
B_f-A_f=(x-y)C_f(x,y).
\]
The boundary integrand therefore contains the factor
\(\varepsilon^{2\kappa+2}\). Its contribution near \((0,0)\) is
bounded by a constant multiple of
\[
\varepsilon^{2\kappa+2}
\int_{\varepsilon}^{\delta}y^{2\alpha}\,dy.
\]
If \(2\alpha>-1\), the integral is uniformly bounded. If
\(2\alpha=-1\), it is bounded by a constant multiple of
\(\log(1/\varepsilon)\). If \(2\alpha<-1\), it is bounded by a
constant multiple of \(\varepsilon^{2\alpha+1}\), so the
corresponding contribution is bounded by
\[
C\varepsilon^{2\alpha+2\kappa+3}
\longrightarrow0
\]
because \(\alpha+\kappa>-3/2\). Away from the vertex, the remaining
integral is uniformly bounded because \(\gamma>-1\). Hence
\[
I_\varepsilon^{(d)}\longrightarrow0.
\]

Combining the three estimates proves the result.
\end{proof}

The next boundary estimate is proved by the same decomposition into the three
sides of \(\partial\triangle_\varepsilon\). The corresponding bounds are
simpler, but they rely on the same endpoint estimates and on the standing
conditions on \(\alpha\), \(\gamma\), and \(\kappa\).

\begin{lemma}\label{lem:ladder-boundary-estimates}
Let \(f\) and \(g\) be symmetric polynomials. Then
\[
\int_{\partial\triangle_\varepsilon}
f(x,y)g(x,y)
(xy)^{\alpha+1}z^\gamma
(x-y)^{2\kappa+2}
\,(dx+dy)
\longrightarrow0
\qquad
\text{as }\varepsilon\to0^+.
\]
\end{lemma}

\begin{proof}
On \(\Gamma_\varepsilon^{(y)}\), we have \(dy=0\), and hence the
absolute value of the corresponding contribution is bounded by
\[
C\varepsilon^{\alpha+1}
\int_{2\varepsilon}^{1-2\varepsilon}
x^{\alpha+1}
(1-x-\varepsilon)^\gamma
(x-\varepsilon)^{2\kappa+2}\,dx.
\]
The integral is uniformly bounded. Indeed, near the origin its
integrand is bounded by a constant multiple of
\(
x^{\alpha+2\kappa+3}
\)
whose exponent is greater than \(-1\), while the other endpoint is
controlled by \(\gamma>-1\). Since \(\alpha>-1\), this contribution
tends to zero.

On \(\Gamma_\varepsilon^{(z)}\), we have \(dy=-dx\). Therefore
\[
dx+dy=0,
\]
and the corresponding boundary contribution vanishes identically.

On \(\Gamma_\varepsilon^{(d)}\), write
\(
x=y+\varepsilon
\). The absolute value of the boundary contribution is bounded by
\[
C\varepsilon^{2\kappa+2}
\int_{\varepsilon}^{(1-2\varepsilon)/2}
[y(y+\varepsilon)]^{\alpha+1}
(1-2y-\varepsilon)^\gamma\,dy.
\]
To obtain a bound uniform in \(\varepsilon\), split the integral at
\(y=1/4\). On \([\varepsilon,1/4]\), the factor
\(1-2y-\varepsilon\) is uniformly separated from zero, and
\(y+\varepsilon\leq2y\). Hence this part is bounded by a constant
multiple of
\[
\int_0^{1/4} y^{2\alpha+2}\,dy.
\]
On \([1/4,(1-2\varepsilon)/2]\), the factor
\([y(y+\varepsilon)]^{\alpha+1}\) is uniformly bounded. Using
\(s=1-2y-\varepsilon\), this part is bounded by a constant multiple of
\[
\int_0^{1/2}s^\gamma\,ds.
\]
Both integrals are finite because \(\alpha,\gamma>-1\). Thus the
original integral is uniformly bounded. Since
\(\kappa>-1\),
\[
\varepsilon^{2\kappa+2}\longrightarrow0.
\]
Thus the contribution from
\(\Gamma_\varepsilon^{(d)}\) also tends to zero.
\end{proof}

\end{document}